\documentclass{conm-p-l}
\usepackage{trajan}
\usepackage{multirow}
\usepackage{latexsym,array,delarray,amsthm,amssymb,epsfig}
\usepackage{amsmath}

\usepackage[utf8]{inputenc} % set input encoding (not needed with XeLaTeX)
\usepackage{amsmath}
\usepackage{youngtab}
\usepackage{mathrsfs}
\usepackage{graphicx} % support the \includegraphics command and options

\usepackage{booktabs} % for much better looking tables
\usepackage{array} % for better arrays (eg matrices) in maths
\usepackage{verbatim} % adds environment for commenting out blocks of text & for better verbatim
\usepackage{subfig} % make it possible to include more than one captioned figure/table in a single float
\usepackage{chngcntr}

\usepackage{fancyhdr} % This should be set AFTER setting up the page geometry
\newtheorem{theorem}{Theorem}

\newtheorem{lemma}{Lemma}

\newtheorem{proposition}{Proposition}

\counterwithin{corollary}{section}
\counterwithin{proposition}{section}
\counterwithin{equation}{section}
\counterwithin{lemma}{section}
\counterwithin{theorem}{section}
\counterwithin{defn}{section}
\counterwithin{conjecture}{section}

\begin{document}
\title[Tensor Product Decomposition and Identities]{Tensor Product Decomposition of $\widehat{\mathfrak{sl}}(n)$ Modules and Identities}
\author{Kailash C. Misra \and  Evan A. Wilson}
\address{Department of Mathematics, North Carolina State University,  Raleigh,  
NC 27695-8205}
\email{misra@ncsu.edu}
\address{
Instituto de Matem\'{a}tica e Estat\'{i}stica\\Universidade de S\~{a}o Paulo}
\email{wilsonea@ime.usp.br}
\thanks{KCM: partially supported by NSA grant, H98230-12-1-0248,\\
\indent EAW: supported by a postdoctoral fellowship from FAPESP (2011/12079-5).} 
\date{} % Activate to display a given date or no date (if empty),
         % otherwise the current date is printed 

%\begin{document}
%\maketitle
\begin{abstract}
We decompose the $\widehat{\mathfrak{sl}}(n)$-module $V(\Lambda_0) \otimes V(\Lambda_i)$ and give generating function identities for the outer multiplicities. In the process we discover some seemingly new partition identities for $n=3, 4$.
\end{abstract}

\maketitle
\bigskip
\section{Introduction}
\par
The affine Lie algebras are the simplest family of infinite dimensional Kac-Moody Lie algebras (cf. \cite{K}). The connection between affine Lie algebra representations and partition identities is well known (for example, see \cite{K1}, \cite{L}, \cite{LM}, \cite{LW}) since 1970's. The affine Lie algebra $\mathfrak{g} = \widehat{\mathfrak{sl}}(n, \mathbb{C})$ is the infinite dimensional analog of the finite dimensional simple Lie algebra $\mathfrak{sl}(n, \mathbb{C})$ of $n \times n$ trace zero matrices. In fact the affine Lie algebra $\widehat{\mathfrak{sl}}(n, \mathbb{C}) = \mathfrak{sl}(n,\mathbb{C}) \otimes \mathbb{C}[t,t^{-1}] \oplus \mathbb{C} c \oplus \mathbb{C} d$ is generated by the degree derivation 
$d=1 \otimes t\frac{d}{dt}$, and the Chevalley generators: 
$$ e_j=E_{j,j+1} \otimes 1,  \qquad f_j=E_{j+1,j} \otimes 1, \qquad  h_j= (E_{jj}-E_{j+1,j+1})\otimes 1,\ \  j=1,2,\dots, n-1, $$
$$ e_0=E_{n,1} \otimes t,\qquad f_0=E_{1,n} \otimes t^{-1}, \qquad h_0=(E_{n,n} - E_{1,1}) \otimes 1 + c,$$
where $c=\sum_{j=0}^{n-1} h_j$ spans the one-dimensional center and $E_{i,j}$ denote the $n \times n$ matrix units. With respect to the
\emph{Cartan subalgebra}  ${\mathfrak{h}}:=\text{span}_{\mathbb{C}}\{h_0, h_1, \dots, h_{n-1}\} \oplus \mathbb{C}d$, let $\Delta =\{\alpha_0, \alpha_1, \dots, \alpha_{n-1}\}$ be the set of simple roots and
$\Phi$ be the set of roots for  $\widehat{\mathfrak{sl}}(n,\mathbb{C})$. Then $\delta = \alpha_0+ \alpha_1+ \cdots + \alpha_{n-1}$ is the null root and  $P = \text{span}_{\mathbb{Z}}\{\Lambda_0,\Lambda_1,\dots, \Lambda_{n-1}, \delta\},$ is the \emph{weight lattice}  where $\langle \Lambda_j, h_k\rangle=\delta_{jk},$ and $\langle \Lambda_j,d\rangle=0$, for $j=0,1,\dots,n-1$.  The  \emph{dominant weight lattice} is defined to be  $P_+=\text{span}_{\mathbb{Z}_{\geq 0}}\{\Lambda_0,\Lambda_1,\dots,\Lambda_{n-1}\}\oplus \mathbb{Z}\delta.$  By the \emph{level} of $\lambda \in P_+$ we mean the nonnegative integer $\text{level}(\lambda)= \lambda (c)$.  For notational convenience we define $\alpha_j = \alpha_{\overline{j}}$ and $\Lambda_j = \Lambda_{\overline{j}}$ for all $j \in \mathbb{Z}$ where  $\overline{j}:=j \pmod{n}$.  

For  $\lambda \in P_+$, let $V(\lambda)$ denote the irreducible integrable $\widehat{\mathfrak{sl}}(n)$-module with highest weight $\lambda$.
For $\lambda , \mu \in P_+$, it is known that the tensor product module $V(\lambda)\otimes V(\mu)$ is completely reducible (cf. \cite[Corollary 10.7 b]{K}), that is:

\begin{equation}
V(\lambda) \otimes V(\mu)=\bigoplus_{\nu \in P_+}c^{\nu}_{\lambda,\mu}V(\nu) \label{gendecomp},
\end{equation}
where $c^{\nu}_{\lambda,\mu}$, called the outer multiplicity, denotes the number of times $V(\nu)$ occurs in this decomposition. In \cite{MW} we studied these outer multiplicities using the crystal base theory for the case $\lambda = \mu = \Lambda_0$ and obtained several identities. In this paper we consider the case $\lambda =  \Lambda_0, \mu = \Lambda_i$, where $1 \le i \le n-1$. Using the approach in \cite{MW} we obtain several different identities, some of them seemingly new for the cases $i = 1, n = 3, 4$.

\section{Decomposition of $V(\Lambda_0)\otimes V(\Lambda_i)$}

In order to decompose the module $V(\Lambda_0)\otimes V(\Lambda_i)$ we will us the theory of crystals (\cite{Ka1}, \cite{Ka2}, \cite{Lu}, \cite{HK}) associated  with integrable representations of quantum affine algebras. Indeed in this paper we will use the explicit realization of the crystal 
$ B(\Lambda_i)$ for the module $V(\Lambda_i)$ in terms of \emph{extended Young diagrams} (or \emph{colored Young diagrams})(\cite{MM}, \cite{JMMO}) which we briefly describe.

Let $I:=\{0,1,\dots n-1\}$ denote the index set for $\widehat{\mathfrak{sl}}(n)$.  An \emph{extended Young diagram} is a collection of $I$-colored boxes arranged in left-justified rows and top-justified columns, such that the number of boxes in each row is greater than or equal to the number of boxes in the row below.  To every extended Young diagram we associate a \emph{charge}, $i\in I$.  In each box, we put a \emph{color} $j\in I$ given by $j \equiv a-b+i \pmod{n}$ where $a$ is the number of columns from the right and $b$ is the number of rows from the top (see figure \ref{pattern}).  

\begin{figure}[bt]
\centering
\begin{tabular}{|c|c|c|c}
\hline
$i$ & $i+1$ & $i+2$ & $\cdots$\\
\hline
$i-1$ & $i$ & $i+1$ & $\cdots$\\
\hline
$i-2$ & $i-1$ & $i$ &  $\cdots$\\
\hline
\vdots & \vdots & \vdots \\
\end{tabular}
\caption{Color pattern for an extended Young diagram of charge $i$.  All labels are reduced modulo $n$.} \label{pattern}
\end{figure}

For example, {\scriptsize\young(120,01)} is an extended Young diagram of charge $1$ for $n=3$. The \emph{null diagram} with no boxes---denoted by $\varnothing$---is also considered an extended Young diagram.

A column in an extended Young diagram is $j$-\emph{removable} if the bottom box contains $j$ and can be removed leaving another extended Young diagram.  A column is $j$-\emph{admissible} if a box containing $j$ could be added to give another extended Young diagram.

An extended Young diagram is called $n$-\emph{regular} if there are at most $(n-1)$ rows with the same number of boxes.   Let $\mathcal{Y}(i)$ denote the collection of all $n$-regular extended Young diagrams of charge $0$.  Then $\mathcal{Y}(i)$ can be given the structure of a crystal with the following actions of $\tilde{e}_j$, $\tilde{f}_j$, $\varepsilon_j,$ $\varphi_j$, and wt$(\cdot)$.  For each $j\in I$ and $b\in \mathcal{Y}(i)$ we define the \emph{$j$-signature} of $b$ to be the string of $+$'s, and $-$'s in which each $j$-admissible column receives a $+$ and each $j$-removable column receives a $-$ reading from right to left.  The \emph{reduced $i$-signature} is the result of recursively canceling all `$+-$' pairs in the $i$-signature leaving a string of the form $(-,\dots, -,+\dots, +)$.  The Kashiwara operator $\tilde{e}_j$ acts on $b$ by removing the box corresponding to the rightmost $-$, or maps $b$ to $0$ if there are no minus signs.  Similarly, $\tilde{f}_j$ adds a box to the bottom of the column corresponding to the leftmost $+$, or maps $b$ to $0$ if there are no plus signs.  The function $\varphi_j(b)$ is the number of $+$ signs in the reduced $j$-signature of $b$ and $\varepsilon_j(b)$ is the number of $-$ signs.  We define 
wt$:\mathcal{Y}(i)\rightarrow P$ by $b \mapsto \Lambda_i-\sum_{j=0}^{n-1}$\# $\{j$-colored boxes in $b\}\alpha_j$.  Then $\mathcal{Y}(i)\cong B(\Lambda_i).$

%\section{Decomposition of $V(\Lambda_0)\otimes V(\Lambda_i)$}

In order to obtain the decomposition of $V(\Lambda_0)\otimes V(\Lambda_i)$, it suffices to find the set of the \emph{maximal} elements of the crystal base $B(\Lambda_0) \otimes B(\Lambda_i)$, i.e. the set of all $b_1\otimes b_2\in B(\Lambda_0) \otimes B(\Lambda_i)$ for which $\tilde{e}_i (b_1\otimes b_2)=0$ for all $i\in I$.  Maximal elements are characterized by the following:
\begin{lemma}[see \cite{JMMO}]
An element $b_1 \otimes b_2 \in B(\Lambda_0)\otimes B(\Lambda_i)$ is maximal if and only if 
$\tilde{e}_jb_1=0 $ and $\tilde{e}_j^{\delta_{j0}+1}b_2=0$ for all $j\in I$.  
\end{lemma}

\begin{lemma}[\cite{MW}]\label{l2}
An element $b_1 \otimes b_2$ of the $U_q(\widehat{\mathfrak{sl}}(n))$ crystal $B(\Lambda_0) \otimes B(\Lambda_i)$ is maximal if and only if $b_1$ is the null diagram and the following two conditions are satisfied for $b_2$:
\begin{enumerate}
\item the first removable column from the right in $b_2$ is $0$-removable,
\item for all $j \in \{0,1,\dots, n-1\}$, if the $k$th admissible column in $b_2$ is $j$-admissible then the $k+1$st removable column, if it exists, is $j$-removable.
\end{enumerate}
\end{lemma} 

As an application of Lemma \ref{l2}, all maximal elements of weight $2\Lambda_0-3\delta$ for $B(\Lambda_0)\otimes B(\Lambda_2)$ with $n=3$ are given in Figure \ref{delta2}.\\
\begin{figure}[tb]
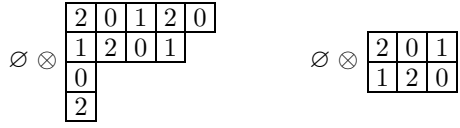

$\varnothing \;\otimes$ \parbox[c]{1in}{\young(20120,1201,0,2)}\qquad $\varnothing \;\otimes$ \parbox[c]{2in}{\young(201,120)}
\caption{Examples of highest weight elements in $B(\Lambda_0)\otimes B(\Lambda_2)$, for $n=3$.}  \label{delta2}
\end{figure}
We denote a \emph{partition} by a finite sequence $(\lambda_1^{f_1},\lambda_2^{f_2},\dots, \lambda_l^{f_l}),$ where $\lambda_k\in \mathbb{Z}_{> 0},$ $\lambda_k>\lambda_{k+1},$ and $f_k\in \mathbb{Z}_{> 0}$ denotes the multiplicity of 
$\lambda_k$.  Each $b\in\mathcal{\mathcal{Y}}(i)$ can be uniquely represented as a partition where $\lambda_k$ is the number of boxes in a given row, and $f_k$ is the number of rows having $\lambda_k$ boxes.  For example, the two diagrams in Figure \ref{delta2} correspond to the partitions $(5, 4,1^2),$ and $(3^2)$.  We can now rephrase Lemma \ref{l2} in terms of partitions as follows.
\begin{lemma}\label{l3}
The highest weight elements of $B(\Lambda_0)\otimes B(\Lambda_i)$ are in a one-to-one correspondence with the set of all partitions $\lambda=(\lambda_1^{f_1},\lambda_2^{f_2},\dots, \lambda_l^{f_l})$ with $f_k< n,$ $k=1,2,\dots, l$,  satisfying the conditions:

\begin{enumerate}
\item $ \lambda_1- f_1 +i \equiv 0\pmod{n}$,
\item $f_k+f_{k+1}+\lambda_{k}-\lambda_{k+1}\equiv0 \pmod{n}$, for $k<l$.
\end{enumerate}
\end{lemma}

\begin{proof} The condition that each $f_k< n$ is equivalent to the condition that $b\in\mathcal{\mathcal{Y}}(i)$ is $n$-regular.  
The first column from the right is always removable, so by condition (1) of Lemma \ref{l2} the first column must contain a 0-colored box  in the bottom row.  Since $\lambda_1$ is the number of columns in the diagram and $f_1$ is the number of boxes in the rightmost column, we see that $\lambda_1-f_1 +i\equiv 0 \pmod{n}$.

Now, suppose that the $k$th admissible column is $j$-admissible, and there exists a removable column to the left of that column.  The $k$th admissible column is $\lambda_{k}+1$ columns from the left and contains $f_1+f_2+\dots +f_{k-1}$ boxes.  Therefore: 
\begin{eqnarray}
f_1+f_2+\dots +f_{k-1}-(\lambda_k+1)+i&\equiv&j-1 \pmod{n}\nonumber\\
f_1+f_2+\dots +f_{k-1}-\lambda_k+i&\equiv&j \pmod{n}\label{t1}
\end{eqnarray}

By condition (2) of Lemma \ref{l2} the $k+1$st removable column (the $\lambda_{k+1}$st from the left) is also $j$-removable, and contains $f_1+f_2+\dots +f_k+f_{k+1}$ boxes.   Therefore:
\begin{equation}
f_1+f_2+\dots+f_k+f_{k+1}-\lambda_{k+1}+i \equiv j \pmod{n}\label{t2}
\end{equation}
Subtracting equation \eqref{t1} from equation \eqref{t2} we obtain condition (2). Furthermore, if the partition satisfies the conditions  (1) and (2), then it is in correspondence with an extended Young diagram as in Lemma \ref{l2}.
\end{proof}
Let $\mathscr{C}_n^i$ be the collection of all partitions satisfying conditions (1) and (2) in Lemma \ref{l3}. Then each such partition in $\mathscr{C}_n^i$ corresponds to a unique maximal element in  $B(\Lambda_0)\otimes B(\Lambda_i)$. Let  
$\mathscr{C}_{n,\mu}^i$ denote the set of elements in $\mathscr{C}_n^i$ of weight $\mu$. It is known that each connected component of the crystal $B(\Lambda_0)\otimes B(\Lambda_i)$ is  the crystal for the corresponding irreducible summand of the module $V(\Lambda_0)\otimes V(\Lambda_i)$. In the following lemma we determine the connected components of the crystal $B(\Lambda_0)\otimes B(\Lambda_i)$.

\begin{lemma}\label{l4}
Each connected component of $B(\Lambda_0)\otimes B(\Lambda_i)$ is isomorphic to \\
$B(\Lambda_t+\Lambda_u-k\delta)$ for some $t,u \in\{0,1,\dots, n-1\}$ such that $t+u \equiv i\pmod{n}$ and $k\in \mathbb{Z}_{\geq 0}$ such that $k\geq t$ if $t\leq i$ and $k\geq t-i$ if $t>i$ .
\end{lemma}
\begin{proof} 
Let $b\otimes b'\in B(\Lambda_0)\otimes B(\Lambda_i)$ be maximal.  Then $\text{wt}(b)=\Lambda_0$, and $b'$ corresponds to a partition $(\lambda_1^{f_1},\lambda_2^{f_2}\dots,\lambda_l^{f_l})\in \mathscr{C}_n$.  We set $\alpha=\Lambda_i-\text{wt}(b')$ and use the weight formula in $B(\Lambda_i)$  to compute $\alpha$:

\begin{eqnarray}
\alpha&=&\sum_{r=1}^{l}\sum_{j_1=1}^{\lambda_{r}}\sum_{j_2=s_{r-1}+1}^{s_{r}}\alpha_{j_1-j_2+i}, \label{trickysum}\\
	&=&\sum_{r=1}^{l}\sum_{j_1=1}^{\lambda_{r}}\sum_{j_2=s_{r-1}+1}^{s_{r}}(2\Lambda_{j_1-j_2+i}-\Lambda_{j_1-j_2+i-1}-\Lambda_{j_1-j_2+1}+\delta_{\overline{j_1-j_2+i},0}\delta),\nonumber
\end{eqnarray}
where $s_j:=\sum_{m=1}^j f_m$.  The sums telescope, leaving:

\begin{equation}
\alpha=
{\sum_{r=1}^l(\Lambda_{\lambda_r-s_r+i}-\Lambda_{i-s_r}+\Lambda_{i-s_{r-1}}-
\Lambda_{\lambda_r-s_{r-1}+i})}+k\delta,\label{withdelta}
\end{equation}
where $k\in \mathbb{Z}_{\geq 0}$ is the number of $\alpha_0$'s in the sum \eqref{trickysum}.  Now consider the sum in \eqref{withdelta}.
\begin{eqnarray*}
\lefteqn{\sum_{r=1}^l(\Lambda_{\lambda_r-s_r+i}-\Lambda_{i-s_r}+\Lambda_{i-s_{r-1}}-
\Lambda_{\lambda_r-s_{r-1}+i})}\\
	&=&\Lambda_{\lambda_1-f_1+i}+\sum_{r=2}^l\Lambda_{\lambda_{r-1}+f_{r-1}+f_r-s_{r}+i}-\sum_{r=1}^l\Lambda_{i-s_{r}}
	+\sum_{r=1}^l\Lambda_{i-s_{r-1}}-\sum_{r=1}^l\Lambda_{\lambda_r-s_{r-1}+i}\\
	&&\qquad \text{since }\lambda_r\equiv \lambda_{r-1}+f_{r-1}+f_{r}\pmod{n} \\
&=&\Lambda_{0}+\sum_{r=2}^l\Lambda_{\lambda_{r-1}-s_{r-2}+i}-\sum_{r=1}^l
	\Lambda_{\lambda_r-s_{r-1}+i}+\sum_{r=1}^l\Lambda_{i-s_{r-1}}
	-\sum_{r=1}^l\Lambda_{i-s_{r}}\\
	&&\qquad \text{since }\lambda_1-f_1+i\equiv 0 \pmod{n}\\
&=&\Lambda_{0}+\sum_{r=1}^{l-1}\Lambda_{\lambda_{r}-s_{r-1}+i}
	-\sum_{r=1}^l\Lambda_{\lambda_r-s_{r-1}+i}
	+\sum_{r=0}^{l-1}\Lambda_{i-s_{r}}-\sum_{r=1}^l\Lambda_{i-s_{r}}\\
&=&\Lambda_{0}+\Lambda_i-\Lambda_{\lambda_l-s_{l-1}+i}-\Lambda_{i-s_l}.
\end{eqnarray*}

By repeated use of condition (2) of Lemma \ref{l3} we can see that $\lambda_l-s_{l-1}-s_l+2i \equiv \lambda_1-s_0-s_1+2i=\lambda_1-f_1 +2i \equiv i  \pmod{n}.$  Therefore, one of the numbers $\overline{\lambda_l-s_{l-1}+i},\overline{i-s_l}$ is in the interval $I=[ i/2 ,(n+i)/2 ]$ and the other is outside this interval.  Let $t$ be the one in $I$  and $u$ be the one outside $I$.  

There are two cases to consider: $\lceil i/2 \rceil \leq t \leq i$ and $i<t\leq \lfloor (n+i)/2 \rfloor$.  Since $\Lambda_t=\Lambda_0+\omega_t$ (\cite{K}, eq. 12.4.3), using  well-known formulas for the fundamental dominant weights $\omega_t$ of $\mathfrak{sl}(n)$ (see for example \cite{H}) we have:
\begin {eqnarray*}
\alpha&=&\Lambda_0+\Lambda_{i}-\Lambda_t-\Lambda_{u}+k\delta\\
	&=&\Lambda_0+\Lambda_0+\frac{1}{n}\left (\sum_{r=1}^{i}r(n-i)\alpha_r+
		\sum_{r=i+1}^{n-1}i(n-r)\alpha_r\right )\\
	&&-\:\Lambda_0-\frac{1}{n}\left (\sum_{r=1}^{t}r(n-t)\alpha_r+
		\sum_{r=t+1}^{n-1}t(n-r)\alpha_r\right )\\
	&&-\:\Lambda_0-\frac{1}{n}\left (\sum_{r=1}^{u}r(n-u)\alpha_r+
		\sum_{r=u+1}^{n-1}u(n-r)\alpha_r\right )+k\delta.
\end{eqnarray*}

In the case $t \leq  i$ we have $u=i-t$.  Therefore:
\begin{eqnarray*}
\alpha&=&-\sum_{r=1}^{u} r\alpha_r-\sum_{r=u+1}^t (i-t)\alpha_r-\sum_{r=t+1}^{i}(i-r)\alpha_r+k\delta\\
	&=&\sum_{r=0}^u(k-r)\alpha_r+\sum_{r=u+1}^t (k-i+t)\alpha_r+\sum_{r=t+1}^i (k-i+r)\alpha_r+\sum_{r=i+1}^{n-1}k\alpha_r
\end{eqnarray*}
We must have  $k\geq i-t$ in order for the coefficient of $\alpha_{u+1}$ to be $\geq 0$.

If $t > i$ then we have $u=n+i-t$.  Therefore:
\begin{eqnarray*}
\alpha&=&-\sum_{r=i+1}^t (r-i)\alpha_r-\sum_{r=t+1}^u (t-i)\alpha_r-\sum_{r=u+1}^{n-1}(n-r)\alpha_r+k\delta\\
	&=&\sum_{r=0}^ik\alpha_r+\sum_{r=i+1}^t (k-r+i)\alpha_r+\sum_{r=t+1}^u (k-t+i)\alpha_r+\sum_{r=u+1}^{n-1}(k-n+r)\alpha_r
\end{eqnarray*}
In this case  $k\geq t-i$ for the coefficient of $\alpha_{t+1}$ to be $\geq 0$.

  Therefore $\text{wt}(b\otimes b')=\Lambda_0+\Lambda_i-\alpha=\Lambda_t+\Lambda_{u}-k\delta$ where $k\geq t-i$.  Hence $b\otimes b'$ is a maximal element of weight $\Lambda_t+\Lambda_{n+i-t}-k\delta$, and the component of $B(\Lambda_0)\otimes B(\Lambda_i)$ containing $b\otimes b'$ is isomorphic to $B(\Lambda_t+\Lambda_{u}-k\delta)$ for some $k$ satisfying the given conditions.
\end{proof}
\begin{theorem} \label{partition}
The $\widehat{\mathfrak{sl}}(n)$-module $V(\Lambda_0)\otimes V(\Lambda_i)$ decomposes as the direct sum\\
 $\bigoplus_{t=\lceil i/2 \rceil}^{i}\bigoplus_{k=i-t}^{\infty} c^{\Lambda_t+\Lambda_{i-t}-k\delta}_{\Lambda_0,\Lambda_i}V(\Lambda_t+\Lambda_{i-t}-k\delta)\oplus \bigoplus_{t=i+1}^{\lfloor(n+i)/2\rfloor}\bigoplus_{k=t-i}^{\infty} c^{\Lambda_t+\Lambda_{n+i-t}-k\delta}_{\Lambda_0,\Lambda_i} \\ V(\Lambda_t+\Lambda_{n+i-t}-k\delta)$ and the outer multiplicities are given by $$c^{\Lambda_t+\Lambda_{u}-k\delta}_{\Lambda_0,\Lambda_i}=|\mathscr{C}_{n,\Lambda_t+\Lambda_{u}-k\delta}^i|,$$
where the absolute value sign denotes the cardinality.
Here, if $c^{\Lambda_t+\Lambda_{u}-k\delta}_{\Lambda_0,\Lambda_i}=0$ then $V(\Lambda_t+\Lambda_u- k\delta)$ does not occur in the decomposition. Furthermore, if $t<i$ and $k=i-t$ (resp. $t \geq i$ and $k=t-i$) then the unique maximal element is an $(i-t) \times (n-t)$ (resp. $(t-i) \times t$) rectangle.
\end{theorem}
\begin{proof}
By Lemma \ref{l3} $|\mathscr{C}_{n,\mu}^i|$ are the outer multiplicities in the decomposition of $V(\Lambda_0)\otimes V(\Lambda_i)$.  Lemma \ref{l4} gives the weights that occur in the decomposition.\\
If two `0' s appeared in the same row or column, then the coefficient of $\alpha_r$ in $\alpha$ would be $>0$ for all $r\in I$, which is ruled out.  If $t\leq i$ then we have $$\text{ht}(\alpha)=\frac{(2k-u)(u+1)}{2}+(k-i+t)(t-u)+\frac{(2k-i+t+1)(i-t)}{2}+k(n-i-1).$$
If $k=i-t$ this equals $(i-t)(n-t).$  If $t>i$ we let $k=t-i$ and compute $$\text{ht}(\alpha)=(t-i)(i+1)+\frac{(t-i-1)(t-i)}{2}+\frac{(t-i-1)(t-i)}{2}=t(t-i).$$
\end{proof}

\section{Generating Functions for Outer Multiplicities}
In this section we consider the generating functions for the outer multiplicities in Theorem \ref{partition} and give explicit formulas for these generating functions.  First we define  the formal series $f$ in the indeterminates $u,v$ as follows: 
\begin{equation*}
f(u,v)=\sum_{k=-\infty}^\infty u^{k(k-1)/2}v^{k(k+1)/2}.
\end{equation*}
It is easy to verify that the function $f(u,v)$ satisfies the following properties: 
\begin{eqnarray}
f(u,v) &=& f (v,u),\\
f(q^r,q^s)&=& q^r f(q^{2r+s},q^{-r}),\label{shiftpower}
\end{eqnarray}

\noindent for integers $r, s$ not both equal to $0$.
Recall the Euler $\varphi$ function  $\varphi(q):=f(-q,-q^2)$. 
In what follows we will be using the well known Jacobi triple product identity (cf. \cite{A}):
\begin{equation}\label{tripleprod)}
f(u,v)=\prod_{j=1}^\infty (1-u^k v^k)(1+u^{k-1} v^k)(1+u^k v^{k-1}).
\end{equation}

Recall that the (formal) character of the highest weight irreducible $\mathfrak{g}$-module $V(\lambda),\lambda\in \mathfrak{h}^*$ is defined by the formal power series $\text{ch}(V(\Lambda))=\sum_{\alpha\in Q^+}\dim(V(\lambda)_{\lambda-\alpha})\\e(\lambda-\alpha),$ where $e(\mu),\mu\in \mathfrak{h}^*$ is an element of the group ring of $\mathfrak{h}^*$ satisfying $e(\mu)e(\nu)=e(\mu+\nu), \mu,\nu\in \mathfrak{h}^*.$  The $q$-character (or principally specialized character) $\text{ch}_q(V(\lambda))$ is defined by making the substitution $e(-\alpha_i)\mapsto q,i\in I$ in $e(-\lambda)\text{ch}(V(\lambda)).$  We have the following $q$-character formulas for the  $\widehat{\mathfrak{sl}}(n)$-modules $V(\Lambda_i), i \in \{0, 1, \dots , n-1\}$ and $V(\Lambda_0+\Lambda_j)$ for $j\in \{ 0,1, \dots, \lfloor n/2 \rfloor\}$ (cf. \cite{M}):
\begin{equation}
\text{ch}_q(V(\Lambda_i))=\prod_{\substack{j>0 \\ j\not \equiv 0 \pmod{n}}}(1-q^j)^{-1}= \frac{\varphi(q^{n})}{\varphi(q)}.
\label{qdim1}
\end{equation}
\begin{equation}
\text{ch}_q(V(\Lambda_0+\Lambda_j))=\frac{\varphi(q^{n})f(-q^{j+1},-q^{n-j+1})}{\varphi(q)^2}.
\label{qdim2}
\end{equation}
  
Now we define the generating function
\begin{equation}
B_t^i(q)=\sum_{k=r}^\infty b_{tk}^iq^{k-r},
\end{equation}
where
\begin{equation*} 
b_{tk}^i=\begin{cases} 
c^{\Lambda_t+\Lambda_{i-t}-k\delta}_{\Lambda_0,\Lambda_i}, \ \ t\in \{\lceil i/2 \rceil ,\dots, i\},\\   
c^{\Lambda_t+\Lambda_{n+i-t}-k\delta}_{\Lambda_0,\Lambda_i},t\in \{i+1,\dots, \lfloor (n+i)/2 \rfloor\},
\end{cases}
\end{equation*}
and $r=|i-t|$. 
%(resp. $t-i$) if $t\leq i$ (resp. $>i$).  
By Theorem \ref{partition} we have:
\begin{multline*}
\text{ch}(V(\Lambda_0)) \text{ch}(V(\Lambda_i))= \sum_{t=\lceil i/2\rceil}^{i}\sum_{k=0}^\infty b_{tk}^i\:\text{ch}(V(\Lambda_t+\Lambda_{i-t}-k\delta))\\
+\:\sum_{t=i+1}^{\lfloor (n+i)/2\rfloor}\sum_{k=0}^\infty b_{tk}^i\:\text{ch}(V(\Lambda_t+\Lambda_{n+i-t}-k\delta))
\end{multline*}
Hence, multiplying both sides by $e(-\Lambda_0-\Lambda_i)$:
\begin{multline*}
e(-\Lambda_0-\Lambda_i)\text{ch}(V(\Lambda_0))\text{ch}(V(\Lambda_i))=\sum_{t=\lceil i/2 \rceil}^{i}\sum_{k=i-t}^{\infty}b_{tk}^i\:\text{ch}(V(\Lambda_t+\Lambda_{i-t}))e(-\Lambda_0-\Lambda_i-k\delta)\\
	+\:\sum_{t=i+1}^{\lfloor (n+i)/2\rfloor}\sum_{k=t-i}^{\infty}b_{tk}^i\:\text{ch}(V(\Lambda_t+\Lambda_{n+i-t}))e(-\Lambda_0-\Lambda_i-k\delta).
\end{multline*}
Now, specializing $e(-\alpha_i)=q$, we obtain:
\begin{multline*}
\text{ch}_q(V(\Lambda_0))\text{ch}_q(V(\Lambda_i))=\sum_{t=\lceil i/2 \rceil}^{ i } q^{(i-t)(n-t)}\text{ch}_q(V(\Lambda_t+\Lambda_{i-t}))\sum_{k=i-t}^\infty b_{tk}^i q^{nk}\\
+\:\sum_{t=i+1}^{ \lfloor (n+i)/2\rfloor} q^{t(t-i)}\text{ch}_q(V(\Lambda_t+\Lambda_{n+i-t}))\sum_{k=t-i}^\infty b_{tk}^i q^{nk}
\end{multline*}
which gives
\begin{multline*}
\frac{\varphi(q^{n})^2}{\varphi(q)^2}=\sum_{t=\lceil i/2 \rceil}^{i}q^{(t-i)(t-n)}\frac{\varphi(q^{n})f(-q^{2t-i+1},-q^{n-2t+i+1})}{\varphi(q)^2}B_t^i(q^{n})\\
+\:\sum_{t=i+1}^{\lfloor (n+i)/2\rfloor}q^{t(t-i)}\frac{\varphi(q^{n})f(-q^{2t-i+1},-q^{n-2t+i+1})}{\varphi(q)^2}B_t^i(q^{n})
\end{multline*}
and hence:
\begin{multline}
\varphi(q^{n})=\sum_{t=\lceil i/2 \rceil}^{i}q^{(t-n)(t-i)}f(-q^{2t-i+1},-q^{n-2t+i+1})B_t^i(q^{n})\label{eq2}\\
	+\:\sum_{t=i+1}^{\lfloor (n+i)/2\rfloor}q^{t(t-i)}f(-q^{n-2t+i+1},-q^{2t-i+1})B_t^i(q^{n}),
\end{multline}

Note that the series $\varphi(q^{n})=\prod_{j=1}^\infty(1-q^{nj})$ has a zero coefficient in front of $q^j$ whenever $j$ is not a multiple of $n$, and similar is the case for  $B_i(q^{n})$.  However, this is not the case for $q^{t(t-i)}f(-q^{2t-i+1},-q^{n-2t+i+1})$.   So the trick is to rearrange the sum to sort the powers of $q$ carefully as we do  below.\\

In the right-hand side of \eqref{eq2} we have:
\begin{eqnarray*}
f(-q^{n-2t+i+1},-q^{2t-i+1})&=&\sum_{j=0}^{n-1} \sum_{\substack{k\in \mathbb{Z}\\ k\equiv j \pmod{n}}}(-1)^k q^{\frac{1}{2}k((n+2)k+4t-2i-n)}\\
	&=&\sum_{j=0}^{n-1}\sum_{m \in \mathbb{Z}}(-1)^{nm+j}q^{\frac{1}{2}(nm+j)((n+2)(nm+j)+4t-2i-n))}\nonumber.
\end{eqnarray*}
We separate out terms involving the index $m$ in the exponent of $q$:
\begin{eqnarray*}
\lefteqn{ (nm+j)\left (\frac{(n+2)(nm+j)+4t-2i-n}{2}\right )}\\
	&=&(nm+j)\left (\frac{(n+2)nm+(n+2)j+4t-2i-n}{2}\right )\\
	&=&nm\left (\frac{(n+2)nm+(n+2)j+4t-2i-n}{2}\right )+j\frac{(n+2)nm}{2}\\
	&&+\:j\left (\frac{(n+2)j-n+4t-2i}{2}\right )\\
	&=&nm\left (\frac{(n+2)nm+2(n+2)j+4t-2i-n}{2}\right )+j\frac{(n+2)j-n+4t-2i}{2},
\end{eqnarray*}
which gives:
\begin{multline}
f(-q^{n-2t+i+1},-q^{2t-i+1})=\\
	\sum_{j=0}^{n-1}(-1)^jq^{\frac{1}{2}j((n+2)j+4t-2i-n)}\left (\sum_{m \in \mathbb{Z}}(-1)^{nm}q^{nm(\frac{1}{2}((n+2)nm+2(n+2)j+4t-2i-n)}\right ).\label{sepeq}
\end{multline}
Taking $q^{1/n}$ in the inner sum in \ref{sepeq} gives a series $\Psi_{tj}^i(q)=f((-1)^nq^r,(-1)^nq^s)$ for $r,s$ satisfying:
\begin{eqnarray*}
s+r&=&n(n+2)\\
s-r&=&2(n+2)j+4t-2i-n.
\end{eqnarray*}
Explicitly:
\begin{equation*}
\Psi_{tj}^i(q)=f\left ((-1)^nq^{\frac{1}{2}n(n+3)-2t+i-(n+2)j },(-1)^nq^{\frac{1}{2}n(n+1)+2t-i+(n+2)j)}\right ).
\end{equation*}
We have:
\begin{equation*}
q^{(t-n)(t-i)}f(-q^{2t+i+1},-q^{n-2t-i+1})=\sum_{j=0}^{n-1} (-1)^j q^{\frac{1}{2}nj(j-1)-n(t-i)+(t+j-i)(t+j)}\Psi_{tj}^i(q^n)
\end{equation*}
and
\begin{equation*}
q^{t(t-i)}f(-q^{2t+i+1},-q^{n-2t-i+1})=\sum_{j=0}^{n-1} (-1)^j q^{ \frac{1}{2}nj(j-1)+(t+j-i)(t+j)}\Psi_{tj}^i(q^n).
\end{equation*}
Thus \eqref{eq2} becomes:
\begin{multline}
\varphi(q^{n})=\sum_{t=\lceil i/2 \rceil}^{i}B_t^i(q^n)\sum_{j=0}^{n-1} (-1)^j q^{\frac{1}{2}nj(j-1)-n(t-i)+(t+j-i)(t+j)}\Psi_{tj}^i(q^n)\label{leq}\\
	+\:\sum_{t=i+1}^{\lfloor (n+i)/2\rfloor}B_t^i(q^n)\sum_{j=0}^{n-1} (-1)^j q^{ \frac{1}{2}nj(j-1)+(t+j-i)(t+j)}\Psi_{tj}^i(q^n)\\
\end{multline}

\emph{Example:} If $n=2$ and $i=0$ then $\lceil i/2 \rceil =0,\lfloor (n+i)/2 \rfloor=1,$ and (\ref{leq}) gives:
\begin{equation*}
\varphi(q^2)=B_0^0(q^2)(\Psi_{00}^0(q^2)-q\Psi_{01}^0(q^2))+B_1^0(q^2)(q\Psi_{10}^0(q^2)-q^4\Psi_{11}^0(q^2))
\end{equation*} 
where
$$
\Psi_{00}^0(q)=f(q^5,q^3),\qquad\Psi_{01}^0(q)=f(q,q^7),
$$
$$
\Psi_{10}^0(q)=f(q^3,q^5),\qquad\Psi_{11}^0(q)=f(q^{-1},q^9).
$$
One can verify that we have
\begin{eqnarray*}
\Psi_{00}^0(q^2)-q\Psi_{01}^0(q^2)&=&1-q-q^3+q^6+q^{10}-\cdots\\
	&=&f(-q,-q^3),
\end{eqnarray*}
and
\begin{eqnarray*}
q\Psi_{10}^0(q^2)-q^4\Psi_{11}^0(q^2)&=&q-q^2-q^4+q^7+q^{11}-\cdots\\
	&=&qf(-q,-q^3),
\end{eqnarray*}
as desired.

On the other hand, if $i=1$ then $\lceil i/2 \rceil =\lfloor (n+i)/2 \rfloor=1,$ and (\ref{leq}) gives:
\begin{equation*}
\varphi(q^2)=B_1^1(q^2)(\Psi_{10}^1(q^2)-q^2\Psi_{11}^1(q^2))
\end{equation*} 
where
$$
\Psi_{10}^1(q)=f(q^4,q^4),\:\Psi_{11}^1(q)=f(1,q^{8}).
$$
However, this is equivalent to:
\begin{equation*}
\varphi(q^2)=B_1^1(q^2)f(-q^2,-q^2)
\end{equation*}
from which one easily sees:
\begin{equation}\label{n2i1}
B_1^1(q)=\frac{\varphi(q)}{f(-q,-q)}=1+q+q^2+2q^3+2q^4+3q^5+4q^6+5q^7+\cdots.
\end{equation}

\emph{Example:} If $n=3,i=0$ then we have:
\begin{eqnarray*}
\varphi(q^3)&=&B_0^0(q^3)(\Psi_{00}^0(q^3)-q\Psi_{01}^0(q^3)+q^{7}\Psi_{02}^0(q^3))\\
&&+\:B_1^0(q^3)(q\Psi_{10}^0(q^3)-q^4\Psi_{11}^0(q^3)+q^{12}\Psi_{12}^0(q^3))
\end{eqnarray*} 
where
$$
\Psi_{00}^0(q)=f(-q^9,-q^6),\Psi_{01}^0(q)=f(-q^4,-q^{11}),\Psi_{02}^0(q)=f(-q^{-1},-q^{16}),
$$
$$
\Psi_{10}^0(q)=f(-q^7,-q^8),\Psi_{11}^0(q)=f(-q^2,-q^{13}),\Psi_{12}^0(q)=f(-q^{-3},-q^{18}).
$$

If $n=3,i=1$ then we have:
\begin{eqnarray*}
\varphi(q^3)&=&B_1^1(q^3)(\Psi_{10}^1(q^3)-q^2\Psi_{11}^1(q^3)+q^{9}\Psi_{12}^1(q^3))\\
&&+\:B_2^1(q^3)(q^2\Psi_{20}^1(q^3)-q^6\Psi_{21}^1(q^3)+q^{15}\Psi_{22}^1(q^3))
\end{eqnarray*}
where
\begin{equation}\label{n3Psi1}
\Psi_{10}^1(q)=f(-q^8,-q^7),\Psi_{11}^1(q)=f(-q^3,-q^{12}),\Psi_{12}^1(q)=f(-q^{-2},-q^{17}),
\end{equation}
\begin{equation}\label{n3Psi2}
\Psi_{20}^1(q)=f(-q^6,-q^9),\Psi_{21}^1(q)=f(-q,-q^{14}),\Psi_{22}^1(q)=f(-q^{-4},-q^{19}).
\end{equation}

The expression $(t+j-i)(t+j)$ appearing in (\ref{leq}) is the only contribution to the exponent of $q$ that possibly has non-zero residue modulo $n$, so we separate the right hand side of (\ref{leq}) into parts having $\overline{(t+j-i)(t+j)}$ equal.  We cyclically permute the variable $j$ by $t$ units to  $\overline{j-t}$, which transforms $\overline{(t+j-i)(t+j)}$ to $\overline{(j-i)j}$. In the new expression, if $j$ is chosen in the interval $(i/2,  (n+i)/2)$ then $(t+\overline{j-t}-i)(t+\overline{j-t})$ gives the same residue modulo $n$ as $(t+\overline{i-j-t}-i)(t+\overline{i-j-t})$ for $j$ in the same interval.  Therefore, we have:
\begin{equation}
\varphi(q^{n})=\sum_{t,j=\lceil i/2 \rceil}^{\lfloor(n+i)/2 \rfloor} B_t(q^n)q^{\overline{(j-i)j}}a_{tj}^i(q^n),\label{lineq}
\end{equation}
where
\begin{equation*}
a_{tj}^i(q)=\begin{cases}
(-1)^{\overline{j-t}}q^{\mu(t,i,\overline{j-t})-(t-i)}\Psi_{t,\overline{j-t}}^i(q) \text{ if }t\in [\frac{i}{2},i] \text{ and }j=\frac{i}{2}
\text{ or }\frac{n+i}{2},\\ \\
(-1)^{\overline{j-t}}q^{\mu(t,i,\overline{j-t})-(t-i)}\Psi_{t,\overline{j-t}}^i(q) + (-1)^{\overline{i-j-t}}q^{\mu(t,i,\overline{i-j-t})-(t-i)}\Psi_{t,\overline{i-j-t}}^i(q)\\
	\qquad  \text{if }t\in [\frac{i}{2},i] \text{ and }j\neq \frac{i}{2},\frac{n+i}{2},\\ \\
(-1)^{\overline{j-t}}q^{\mu(t,i,\overline{j-t})}\Psi_{t,\overline{j-t}}^i(q) \text{ if }t\in [i+1,\frac{n+i}{2}] \text{ and }j=\frac{i}{2} \text{ or }\frac{n+i}{2},\\ \\
(-1)^{\overline{j-t}}q^{\mu(t,i,\overline{j-t})}\Psi_{t,\overline{j-t}}^i(q)+(-1)^{\overline{i-j-t}}q^{\mu(t,i,\overline{i-j-t})}\Psi_{t,\overline{i-j-t}}^i(q) \text{ otherwise,}
\end{cases}
\end{equation*}
and
\begin{equation*}
\mu(t,i,k)=\frac{k(k-1)}{2}+\left \lfloor\frac{(t+k-i)(t+k)}{n}\right \rfloor.
\end{equation*}

Unfortunately, the numbers $\overline{(j-i)j}$--appearing as residues of exponents of $q$ in (\ref{lineq}) may not all be distinct for all integers $j \in [i/2 , (n+i)/2]$, depending on $n$ and $i$.  In fact, they are all distinct if and only if $(j-i)j \equiv (j'-i)j' \pmod{n}$ implies $j' \equiv i- j,j \pmod{n}$ for all $j,j'$, i.e. if and only if $(j'-j)(j'+j-i)\equiv 0 \pmod{n}$ has only trivial solutions modulo $n$.  The proof of the following is elementary, but we include it for the convenience of the reader.

\begin{proposition}\label{propmod}
The congruence $(j'-j)(j'+j-i)\equiv 0 \pmod{n}$ has only the trivial solutions $j'\equiv j \pmod{n}$ and $j'\equiv i-j \pmod{n}$ if and only if $n$ and $i$ satisfy one of the following conditions:
\begin{enumerate}
\item $i$ is even and $n$ is prime or twice an odd prime,
\item $i$ is odd and $n$ is prime or a power of 2.
\end{enumerate}
\end{proposition}
\begin{proof}
Suppose that $n$ is a composite integer with factorization $n=rs$.  If $i$ is even then $j'=r+s+i/2,j=s-r+i/2$ gives a solution that is trivial if and only if $n=2s$ or $n=2r$.  Since the factorization of $n$ was chosen arbitrarily, we deduce that $n=2p$ for some prime $p$ (and by examining the cases, we can rule out $n=4$).  If $i$ is odd, and $n$ is not a power of $2$, then we can choose a factorization such that $r$ is odd.  In such case $j'=s+(r+i)/2, j=-s+(r+i)/2$ gives a non-trivial solution. 

Conversely, if $n$ is prime then $n|(j'-j)(j'+j-i)$ implies $n|(j'-j)$ or $n|(j'+j-i)$, i.e. any solution is trivial.  If $n=2p$ for an odd prime $p$, and $i$ is even then $2p|(j'-j)(j'+j-i)$ implies that $2$ is a factor of both $j'-j$ and $j'+j-i$, since both have the same parity, and at least one factor is divisible by $p$.  These must be the same factor, since $p$ is odd.  If $i$ is odd and $n=2^t$, then $2^t|(j'-j)(j'+j-i)$ implies that $2^t|(j'-j)$ or $2^t | (j'+j-i)$ since $j'-j$ and $j'+j-i$ have different parity, which finishes the proof.
\end{proof}

Now assume that $i,n$ satisfy one of the two conditions in Proposition \ref{propmod}.  In this case (\ref{lineq}) is equivalent to the following set of linear equations:
\begin{equation*}
q^{\overline{(j-i)j}}\sum_{t=\lceil i/2 \rceil}^{\lfloor (n+i)/2 \rfloor }B_t^i(q^n)a_{tj}^i(q^n)=\delta_{ij}\varphi(q^n), j=\left \lceil\frac{ i}{2} \right \rceil, \dots, \left \lfloor \frac{n+i}{2} \right \rfloor,
\end{equation*}
or,
\begin{equation*}
\sum_{t=\lceil i/2 \rceil}^{\lfloor (n+i)/2 \rfloor }B_t^i(q)a_{tj}^i(q)=\delta_{ij}\varphi(q), j=\left \lceil\frac{i}{2} \right \rceil, \dots, \left \lfloor \frac{n+i}{2} \right \rfloor.
\end{equation*}
which can be written in matrix form as 
\begin{equation*}
A^i(q)\mathbf{B}^i(q)=\mathbf{\Phi}^i(q)
\end{equation*}
where
\begin{multline*}
A^i(q)^T=(a_{tj}^i(q))_{t,j\in I}, \mathbf{B}^i(q)=(B_j^i(q))_{j\in I}^T, \mathbf{\Phi}^i(q)=(\delta_{ij}\varphi(q))_{j\in I}^T,\\
I=\left \{ \left \lceil\frac{i}{2} \right \rceil, \dots, \left \lfloor \frac{n+i}{2} \right \rfloor \right \}.
\end{multline*}

Therefore, Cramer's rule yields the following proposition.
%expression for $B_t^i(q)$:
\begin{proposition} \label{deteqn} For $0 \le i \le n-1$, $\left \lceil\frac{i}{2} \right \rceil \le t \le \left \lfloor \frac{n+i}{2} \right \rfloor$, we have
\begin{equation*}
B_t^i(q)=\frac{(-1)^{i+t}\varphi(q)\det(\widetilde{A^i(q)}_{it})}{\det (A^i(q))},
\end{equation*}
where $\widetilde{A^i(q)}_{it}$ denotes the matrix $A^i(q)$ with the $i$th row and $t$th column deleted.
\end{proposition}

\section{Examples and Identities}

Comparing the result in Theorem \ref{partition} and Proposition \ref{deteqn}  we now have the following theorem which gives generating function identities.

\begin{theorem}\label{th:main}
Let $i$ be as in Proposition \ref{propmod} for $n\geq 2.$  Then, for $t\in \{\lceil i/2 \rceil,\dots,\lfloor (n+i)/2 \rfloor \}$:
\begin{equation}\label{eq:main}
\sum_{k=|i-t|}^\infty|\mathscr{C}^n_{\Lambda_t+\Lambda_u-k\delta}|q^{k-|i-t|}=
\frac{(-1)^{i+t}\varphi(q)\det(\widetilde{A^i(q)}_{it})}{\det (A^i(q))},
\end{equation}
where $u = i-t$ for $t \le i$ and $u = n+i-t$ for $t > i$.
\end{theorem}
\begin{proof}
The left side and the right side of (\ref{eq:main}) both count the outer multiplicity of $V(\Lambda_t+\Lambda_u+k\delta)$ in the decomposition of $V(\Lambda_0)\otimes V(\Lambda_i)$ by Theorem \ref{partition} and Proposition \ref{deteqn} respectively.
\end{proof}
In \cite{MW}, we considered the case $i = 0$ for $n = 2, 3$ and showed that we obtain certain identities in the Slater list \cite{S} and some new identities for $i=0$ and $n=3$.
In this paper we consider the case $i=1$ for $n = 2, 3, 4$ and obtain some seemingly new identities. 

In the case $n=2,i=1$, as we have seen in (\ref{n2i1}):
\begin{equation*}
B_1^1(q)=\frac{\varphi(q)}{f(-q,-q)}=\frac{\prod_{j=1}^\infty(1-q^j)}{\prod_{j=1}^\infty(1-q^{2j})(1-q^{2j-1})^2}=\frac{1}{\prod_{j=1}^\infty(1-q^{2j-1})}
\end{equation*}
 However, the set $\mathscr{C}_{2, \Lambda_1 + \Lambda_0 -k\delta}^1$ is the set of partitions of $2k$ into distinct even parts which is the same as the number of partitions of $k$ into distinct parts.  Thus Theorem \ref{th:main} in this case becomes the Euler's identity (see \cite{A}). This agrees with the corresponding result given in \cite{F}.

We now consider the case $i=1, n=3$. In this case, $\lceil i/2 \rceil = 1, \lfloor(n+i)/2\rfloor=2$.  Using  (\ref{shiftpower}),  (\ref{n3Psi1}), and (\ref{n3Psi2}),
%\begin{equation*}
%\Psi_{10}^1(q)=f(-q^8,-q^7),\Psi_{11}^1(q)=f(-q^3,-q^{12}),\Psi_{12}^1(q)=f(-q^{-2},-q^{17}),
%\end{equation*}
%\begin{equation*}
%\Psi_{20}^1(q)=f(-q^6,-q^9),\Psi_{21}^1(q)=f(-q,-q^{14}),\Psi_{22}^1(q)=f(-q^{-4},-q^{19}).
%\end{equation*}
the matrix $A^1(q)$ in Proposition \ref{deteqn} is:
\begin{eqnarray*}
A^1(q) &=& \begin{pmatrix}
\Psi_{10}^1(q)+q^3\Psi_{12}^1(q)&q^5\Psi_{22}^1(q)-q^2\Psi_{21}^1(q)\\
-\Psi_{11}^1(q)&\Psi_{20}^1(q)
\end{pmatrix}\\
&=&
\begin{pmatrix}
f(-q^8,-q^7)+q^3f(-q^{-2},-q^{17})&q^5f(-q^{-4},-q^{19})-q^2f(-q,-q^{14})\\
-f(-q^3,-q^{12})&f(-q^6,-q^9)
\end{pmatrix}\\
&=&
\begin{pmatrix}
f(-q^8,-q^7)+qf(-q^{2},-q^{13})&qf(-q^{4},-q^{11})-q^2f(-q,-q^{14})\\
-f(-q^3,-q^{12})&f(-q^6,-q^9)
\end{pmatrix}\\
%&&\text{\qquad (where we have used \ref{shiftpower}).}
\end{eqnarray*}
Now, using Frank Garvan's Maple $q$-series package (\cite{G}) we see that  $\det (A^1(q)) = \varphi(q)^2$.  Therefore, Proposition \ref{deteqn} gives the following $q$-series for the outer multiplicities (using (\ref{tripleprod)})):

\begin{eqnarray*}
B_1^1(q)&=&\frac{f(-q^6,-q^9)}{\varphi(q)}\\
&=&\frac{\prod_{j=1}^{\infty}(1-q^{15j})(1-q^{15j-6})(1-q^{15j-9})}{\prod_{j=1}^{\infty}(1-q^j)}\\
B_2^1(q)&=&\frac{f(-q^3,-q^{12})}{\varphi(q)}\\
&=&\frac{\prod_{j=1}^{\infty}(1-q^{15j})(1-q^{15j-3})(1-q^{15j-12})}{\prod_{j=1}^{\infty}(1-q^j)}
\end{eqnarray*}
Now Theorem \ref{th:main} gives the following identities:
\begin{equation}
\frac{\prod_{j=1}^{\infty}(1-q^{15j})(1-q^{15j-6})(1-q^{15j-9})}{\prod_{j=1}^{\infty}(1-q^j)} = \sum_{k=0}^\infty a(k) q^k,
\end{equation}
\begin{equation}
 \frac{\prod_{j=1}^{\infty}(1-q^{15j})(1-q^{15j-3})(1-q^{15j-12})}{\prod_{j=1}^{\infty}(1-q^j)} = \sum_{k=0}^\infty b(k) q^k,
\end{equation}
where $a(k)$ (respectively $b(k)$) is the number of partitions $(\lambda_1^{f_1},\lambda_2^{f_2},\dots, \lambda_l^{f_l})$ of $3k$ (respectively $3k+2$) with $f_1 - \lambda_1 +1 \equiv 0\pmod{3}$,
$f_j+f_{j+1}+\lambda_{j}-\lambda_{j+1}\equiv0 \pmod{3}$,  $f_j <3$ for $1\leq j<l$.
 
%\begin{corollary}
%The number of 3-regular partitions $(\lambda_1^{f_1},\lambda_2^{f_2},\dots, \lambda_l^{f_l})$ of $3k$ with $f_1 - \lambda_1 +1 \equiv 0\pmod{3}$,$f_j+f_{j+1}+\lambda_{j}-\lambda_{j+1}\equiv0 \pmod{3}$, for $1\leq j<l$ is equal to the number of partitions of $k$ into parts $\not\equiv0, \pm 5 \pmod{15}$.
%
%Similarly, the number of partitions of $3(k-1)+2$ satisfying the same conditions is equal to the number of partitions of $k-1$ into parts $\not\equiv0, \pm 3 \pmod{15}$.
%\end{corollary}

Now we consider the case $i=1$ and $n=4$ which satisfies the conditions in 
Proposition \ref{propmod}.  In this case we have $\lceil i/2 \rceil = 1, \lfloor(n+i)/2\rfloor=2$ and
\begin{equation*}
\Psi_{10}^1(q)=f(q^{13},q^{11}),\Psi_{11}^1(q)=f(q^{7},q^{17}),\Psi_{12}^1(q)=f(q,q^{23}),\Psi_{13}^1(q)=f(q^{-5},q^{29})
\end{equation*}
\begin{equation*}
\Psi_{20}^1(q)=f(q^{11},q^{13}),\Psi_{21}^1(q)=f(q^{5},q^{19}),\Psi_{22}^1(q)=f(q^{-1},q^{25}),\Psi_{23}^1(q)=f(q^{-7},q^{31}).
\end{equation*}
Hence we have:
\begin{eqnarray*}
A^1(q)&=&
\begin{pmatrix}
\Psi_{10}^1(q)-q^6\Psi_{13}^1(q)&-q^8\Psi_{23}^1(q)+q^4\Psi_{22}^1(q)\\
-\Psi_{11}^1(q)+q^{2}\Psi_{12}^1(q)&\Psi_{20}^1(q)-q\Psi_{21}^1(q)
\end{pmatrix}
\\
&=&
\begin{pmatrix}
f(q^{13},q^{11})-q^6f(q^{-5},q^{29})&-q^8f(q^{-7},q^{31})+q^4f(q^{-1},q^{25})\\
-f(q^7,q^{17})+q^2f(q,q^{23})&f(q^{11},q^{13})-qf(q^5,q^{19})
\end{pmatrix}\\
&=&
\begin{pmatrix}
f(q^{13},q^{11})-qf(q^{5},q^{19})&-qf(q^{7},q^{17})+q^3f(q,q^{23})\\
-f(q^7,q^{17})+q^2f(q,q^{23})&f(q^{11},q^{13})-qf(q^5,q^{19})
\end{pmatrix}
\end{eqnarray*}

Using Frank Garvan's $q$-series package we see that $\det(A^1(q))=\varphi(q)f(-q,-q).$  Therefore:
\begin{eqnarray*}
B_1^1(q)&=&\frac{f(q^{11},q^{13})-qf(q^5,q^{19})}{f(-q,-q)}\\
B_2^1(q)&=&\frac{f(q^7,q^{17})-q^2f(q,q^{23})}{f(-q,-q)}
\end{eqnarray*}

Hence by Theorem \ref{th:main} we have the following identities:
\begin{equation}
\frac{f(q^{11},q^{13})-qf(q^5,q^{19})}{f(-q,-q)} = \sum_{k=0}^\infty c(k) q^k,
\end{equation}
and
\begin{equation}
\frac{f(q^7,q^{17})-q^2f(q,q^{23})}{f(-q,-q)} = \sum_{k=0}^\infty d(k) q^k,
\end{equation}
where $c(k)$ (respectively $d(k)$) is the number of partitions $(\lambda_1^{f_1},\lambda_2^{f_2},\dots, \lambda_l^{f_l})$ of $4k$ (respectively $4k+2$) with $f_1 - \lambda_1 +1 \equiv 0\pmod{4}$, $f_j+f_{j+1}+\lambda_{j}-\lambda_{j+1}\equiv0 \pmod{4}$, $f_j <4$, for $1\leq j<l$.

\bibliographystyle{amsplain}

\end{document}